\documentclass[reqno, a4paper, 10pt]{amsart}
\usepackage{amssymb}
\usepackage{mathrsfs}
\usepackage{amssymb, url, color, graphicx, amscd, mathrsfs}
\usepackage[colorlinks=true, bookmarks=true, pdfstartview=FitH, pagebackref=true, linktocpage=true, linkcolor = magenta, citecolor = blue]{hyperref}
\usepackage{graphicx}
\usepackage{times}

\newtheorem{theorem}{Theorem}[section]
\newtheorem{lemma}[theorem]{Lemma}
\newtheorem{corollary}[theorem]{Corollary}

\theoremstyle{definition}

\theoremstyle{remark}
\newtheorem{remark}[theorem]{Remark}

\numberwithin{equation}{section}

\allowdisplaybreaks
\begin{document}

\title{The number of cusps of complete Riemannian manifolds with finite volume}

\def\cfac#1{\ifmmode\setbox7\hbox{$\accent"5E#1$}\else\setbox7\hbox{\accent"5E#1}\penalty 10000\relax\fi\raise 1\ht7\hbox{\lower1.0ex\hbox to 1\wd7{\hss\accent"13\hss}}\penalty 10000\hskip-1\wd7\penalty 10000\box7 }

\author[N.T. Dung]{Nguyen Thac Dung}
\address[N.T. Dung]{Department of Mathematics, Mechanics and Informatics,\\ Hanoi University of Sciences, Ha N\^{o}i, Vi\^{e}t Nam}
\email{\href{mailto: N.T. Dung <dungmath@gmail.com>}{dungmath@gmail.com}}

\author[N.N. Khanh]{Nguyen Ngoc Khanh}
\address[N.N. Khanh]{Department of Mathematics, Mechanics and Informatics,\\ Hanoi University of Sciences, Ha N\^{o}i, Vi\^{e}t Nam}
\email{\href{mailto: N.N. Khanh <khanh.mimhus@gmail.com>}{khanh.mimhus@gmail.com}}

\author[T. C. Son]{Ta Cong Son}
\address[T. C. Son]{Department of Mathematics, Mechanics and Informatics,\\ Hanoi University of Sciences, Ha N\^{o}i, Vi\^{e}t Nam}
\email{\href{mailto: T. C. Son <congson82@hus.edu.vn>}{congson82@hus.edu.vn}}

\begin{abstract}In this paper, we will count the number of cusps of
complete Riemannian manifolds $M$ with finite volume. When $M$ is a
complete smooth metric measure spaces, we show that the number of
cusps in bounded by the volume $V$ of $M$ if some geometric
conditions hold true. Moreover, we use the nonlinear theory of the
$p$-Laplacian to give a upper bound of the number of cusps on
complete Riemannian manifolds. The main ingredients in our proof are
a decay estimate of volume of cusps and volume comparison theorems.

\end{abstract}

\subjclass[2010]{Primary 53C23; Secondary 53C24, 58J0} 

\keywords{Cusps; Decay estimate; $p$-Laplacian; Smooth metric measure spaces; Volume comparison theorem}

\maketitle

\section{Introduction}
Let $E$ be an end of a Riemannian manifold $M^n$ and $\lambda_1(M)$
be the first Dirichlet eigenvalue of the Laplacian on $M$. It is well-known
that information of $\lambda_1(M)$ tells us some geometric
properties of the manifold. For example, if $\lambda_1(M)>0$ then
$M$ must have infinite volume, or if $\lambda_1(E)>0$ then either
$E$ has finite volume, namely $E$ is a cusp; or $E$ is
non-parabolic end with volume of exponent growth. In \cite{Cheng}, Cheng considered
complete manifolds $(M^n, g)$ of dimension $n$ with $Ric_M\geq-(n-1)$ and gave an upper
bound of $\lambda_{1}(M)$
$$\lambda_1(M)\leq\frac{(n-1)^2}{4}.$$
Later, Li and Wang showed in \cite{LW02} that if $Ric_M\geq -(n-1)$ and $\lambda_1(M)$ is maximal then
either $M$ has only one end or; $M$ is a topological cylinder with certain warped metric
product. Since $\lambda_1(M)$ is maximal, $M$ must have infinite volume. Hence, Li-Wang's result
says that we can count the number of ends of complete Riemannian manifold $M$ with $Ric_M\geq-(n-1)$
provided $\lambda_1(M)$ obtains its maximal value. In this case,$M$ has at most two ends.

Interestingly,
when $M$ is a complete Riemannian manifold of finite volume, Li and Wang proved in \cite{LW2010} that one can count ends (cusps) via the bottom of Neumann spectrum $\mu_1(M)$ defined by
$$\mu_1(M)=\inf\limits_{\phi\in H^1(M)), \int_M\phi=0}\dfrac{\int_M|\nabla\phi|^2}{\int\phi^2}.$$
Note that $\mu_1(M)$ plays the role of a generalized first non-zero Neumann eigenvalue, although $\mu_1(M)$ might not necessarily be an eigenvalue. However, if $M$ is compact then $\mu_1(M)$  is itself a positive eigenvalue called the first (Neumann) eigenvalue of the Laplacian (see \cite{Veron} or also see \cite{Matei}). As in \cite{Veron, LW2010}, variational principle implies
$$\mu_1(M)\leq\max\{\lambda_1(\Omega_1), \lambda_1(\Omega_2)\}$$
for any two disjoint domains $\Omega_1$ and $\Omega_2$ of $M$, where $\lambda_1(\Omega_1)$ and $\lambda_1(\Omega_2)$ are their first Dirichlet eigenvalues respectively. Li Wang counted cusps on complete
manifolds with finite volume as follows.
\begin{theorem}[\cite{LW2010}]\label{l1}
Let $M^n$ be a complete Riemannian manifold with Ricci curvature bounded from below by
$Ric_M\geq-(n-1)$
Assume that $M$ has finite volume given by $V$, and
$$\mu_1(M)\geq\dfrac{(n-1)^2}{4}.$$
Let us denote $N(M)$ to be the number of ends (cusps) of M. Then there exists a constant $C(n)>0$ depending only on $n$, such that,
$$N(M)\leq C(n)\Big(\frac{V}{V_o(1)}\Big)^2\ln\frac{V}{V_o(1)}$$
where $V_o(1)$ denotes the volume of the unit ball centered at any point $o\in M$.
\end{theorem}
Due to Cheng's upper bound of $\lambda_1(M)$ and variational characteristic of $\mu_1(M)$, one can see that in Theorem \ref{l1}, $\mu_1(M)$ is maximal. This means that instead of $\lambda_1(M)$, we can use $\mu_1(M)$ to count the number of cusps of complete Riemannian manifolds with finite volume. We would like to mention that to count cusps of complete manifolds, Li and Wang used a decay volume estimate and a volume comparison theorem.

Motivated by the beauty of Theorem \ref{l1}, in this paper, we want to estimate the number of cusps on smooth metric measure spaces $(M^n, g, e^{-f}dv)$ with finite $f$-volume $V_f$. Recall that a smooth metric measure space $(M^n, g, e^{-f}dv)$ is a complete Riemannian manifold $(M^n, g)$ of dimension $n$ with $f\in\mathcal{C}^\infty(M)$ is a smooth weighted function and $e^{-f}dv$ is the weighted volume. Here $dv$ is the volume form with respect to the metric $g$. On $(M, g, e^{-f}dv$, we consider the weighted Laplacian
$$\Delta_f\cdot=\Delta\cdot+\left\langle\nabla f, \nabla\cdot\right\rangle$$
which is a self-adjoint operator. Associated to the weighted Laplacian, we define the Bakry-\'{E}mery curvature by
$$Ric_f:=Ric_M+Hess f$$
where $Ric_M$ is the Ricci curvature of $M$ and $Hess f$ is the Hessian of $f$. Following the same strategy as in \cite{LW2010}, we first give a decay estimate for weighted volume and use a volume comparison theorem in \cite{MW} to prove the below theorem.
\begin{theorem}\label{main3} Let $(M,g,e^fdv)$ be a complete smooth metric measure space. Assume for some nonnegative constants $\alpha$ and $\beta$, $$ |f|(x)\leq \alpha r(x)+\beta$$
for $x\in M$ with Ricci curvature bounded from below by
$$Ric_f\geq-(n-1)$$
Assume that $M$ has finite volume given by $V_f$, and $$\mu_1(M)\geq\dfrac{(n-1+\alpha)^2}{4}$$
Let us denote $N(M)$ to be the number of ends (cusps) of $M$. Then there exists a contant $C(n)>0$ depending only on $n$, such that,
$$N(M)\leq C(n)\Big(\dfrac{V_f}{V_{o,f}(1)}\Big)^2\ln\Big(\dfrac{V_f}{V_{o,f}(1)}\Big)$$
where $V_{o,f}(1)$ denotes the $f$-volume of the unit ball centered at any fixed point $o\in M$.
\end{theorem}

On the other hand, our second aim in this paper is to count the number of cusps on complete Riemannian manifold $(M^n, g)$ with finite volume via the nonlinear theory of $p$-Laplacian, for $1\leq p<\infty$. On a complete Riemannian manifold, for any $u\in W^{1,p}_{loc}$, the $p$-Laplacian denoted by $\Delta_p$ acting on $u$ as follows
$$\Delta_pu=\operatorname{div}(|\nabla u|^{p-2}\nabla u).$$
If $\lambda$ satisfies
$$\Delta_pu=-\lambda |u|^{p-2}u$$
then $\lambda$ is said to be an eigenvalue of $p$-Laplacian and $u$ is called an eigenfunction with respect to $\lambda$. As in \cite{Matei} and \cite{Veron}, we define
$$\mu_{1, p}(M)=\inf\left\{\frac{\int_M|\nabla\phi|^p}{\int\phi^p}; \phi\in W^{1,p}(M), \phi\not=0\text{ and }\int_M|\phi|^{p-1}\phi=0\right\}.$$
For any couple of nonempty disjoint open subsets $\Omega_1, \Omega_2$ of $M$, Veron proved, in his paper \cite{Veron}, (see also \cite{Matei}) the following result
$$\mu_{1, p}(M)\leq\max\{\lambda_{1,p}(\Omega_1), \lambda_{1,p}(\Omega_2)\}$$
where $\lambda_{1,p}(\Omega_1)$ and $\lambda_{1,p}(\Omega_2)$ are their first Dirichlet $p$-eigenvalues, respectively. As in \cite{Dung, SW} we know that if $M$ satisfies $Ric_M\geq-(n-1)$ then $\lambda_{1, p}$ is bounded from upper by
$$\lambda_{1, p}(M)\leq\big(\frac{n-1}{p}\Big)^p.$$
For further discussion on $p$-Laplacian and its eigenvalues, we refer the reader to \cite{Dung, Matei, SW, Veron} and the references therein. Now, we can count the number of cusps as follows.
\begin{theorem}\label{12main12}
Let $(M^n,g)$ be a complete with finite volume given by $V$. Assume
that the Ricci curvature is bounded from below by
$$Ric_M\geq-(n-1)$$
and
 $$\mu_{1,p}(M)\geq\dfrac{(n-1)^p}{p^p}.$$
 Let us denote
$N(M)$ to be the number of ends (cusps) of $M$. Then  there exists a contant $C(n,p)>0$ depending only
on $n$ and $p$ such that,
$$N(M)\leq C(n,p)\Big(\frac{V}{V_{o}(1)}\Big)^2\ln^{\alpha(p)} \Big(\frac{V}{V_{o}(1)}\Big)$$
where  $\alpha(p)=\max\{1, p/2\}$ and $V_o(1)$ denotes the volume of the unit ball centered at a fixed point $o\in M$. 
\end{theorem}
As noticed in the previous part, two main ingredients in our proof are a decay estimate for volume and a comparison volume theorem. Moreover, it is also worth to note that $\mu_{1, p}(M)$ in Theorem \ref{12main12} is maximal. When $\lambda_{1, p}(M)$ is maximal, it is proved in \cite{Dung, SW} that $M$ has at most two ends. In fact, in \cite{Dung, SW}, the authors pointed out that either $M$ must have only one end or; $M$ is a topological cylinder provided that $Ric_M\geq-(n-1)$ and $\lambda_{1,p}$ is maximal. 

The paper is organized as follows. In the Section \ref{section2}, we introduce a estimate of volume decay rate on smooth metric measure spaces which can be considered as a generalization of volume decay rate in \cite{LW2006, LW2010}. In the Section \ref{section3}, we will count the number of cusps on complete smooth metric measure spaces with finite volume. Finally, we use the nonlinear theory of $p$-Laplacian in Section \ref{section4} to estimate the number of cusps on Riemannian manifolds. Several complete manifolds of finite volume are investigated in this section.
\section{Smooth metric measure spaces with weighted Poincar\'{e} inequality}\label{section2}
Let $(M^n, g, e^{-f}dv)$ be a smooth metric measure space with a weighted positive function $\rho\in\mathcal{C}(M)$. We define the $\rho$-metric by
$$ d{s_\rho}^2=\rho ds^2.$$
Thanks to this metric, we can define the $\rho-$distance function to be
$$ r_\rho(x,y)=\inf_\gamma \ell_\rho(\gamma),$$
where the infimum is taken over all smooth curves $\gamma$ joining $x$ and $y$, and $\ell$ is the length of $\gamma$ with respect to $ds^2_\rho$. For a fixed point $o\in M$, we denote $r_\rho(x)=r_\rho(o,x)$ to be the $\rho-$distance to $o$. As in \cite{LW2006}, we know that $|\nabla r_\rho|^2(x)=\rho(x)$.

Throughout this article, we denote
$$ B_\rho(o,R)=\{x\in M| r_\rho(o, x)<R\}$$
to be the geodesic ball centered at $o\in M$ with radius $R$.
We also denote the geodesic ball
$$ B(o,R)=\{x\in M|r(o,x)<R\}$$
to be the set of points in $M$ that has distance less than $R$ from point $o$ with respect to the background metric $ds^2_M$. For simple the notation, sometime, we will suppress the dependency of $o$ and write $B_\rho(R)=B_\rho(o,R)$ and $B(R)=B(o,R)$. Finally, suppose that $E$ is an end of $M$, we denote $E_\rho(R)=B_\rho(R)\cap E$.
\begin{lemma}\label{1}
Let $(M,ds^2_\rho,e^{-f}dv)$ be a complete smooth metric measure space.
Suppose $E$ is an end of $M$ satisfying there exists a nonnegative function $\rho(x)$ defined on $E$ with the property that
$$ \int_E\rho\phi^2e^{-f}\leq \int_E|\nabla\phi|^2e^{-f}-\int_E\mu\phi^2e^{-f}$$
for any compactly supported function $\phi\in C_c^\infty(E)$ and $\mu$ be a function defined on $E$. Let $u$ be a nonnegative function defined on $E$ such that the differential inequality
$$ \Delta_fu\geq -\mu u$$
holds true. If $u$ has the growth condition
$$ \int_{E_\rho(R)}\rho u^2e^{-2r_\rho}e^{-f}=o(R)$$
as $R\to \infty$, then it must satisfy the decay estimate
$$ \int_{E_{\rho(R+1)}\setminus E_{\rho(R)}}\rho u^2e^{-f}\leq C(1+(R-R_0)^{-1})e^{-2R}\int_{E_\rho(R_0+1)\setminus E_\rho(R_0)}e^{2r_\rho}u^2e^{-f}$$
for some constant $C>0$.
\end{lemma}
\begin{proof}[Proof of Lemma \ref{1}]To prove the Lemma \ref{1}, we will combine both arguments in \cite{LW2006} and \cite{LW2010}. Let $\phi(r_\rho(x))$ be a nonnegative cut-off function with support in $E$, where $r_\rho(x)$ is the $\rho-$distance to the fixed point $p$. Then for any function $h(r_\rho(x))$, integration by parts implies
\begin{align}
&\int_E|\nabla(\phi e^{2h}u)|^2e^{-f}\notag\\
=&\int_E|\nabla(\phi e^h)|^2u^2e^{-f}+\int_E(\phi e^h)^2|\nabla u|^2e^{-f}+2\int_E(\phi e^h)u\langle \nabla(\phi e^h),\nabla u\rangle e^{-f}\notag\\
=&\int_E|\nabla(\phi e^h)|^2u^2e^{-f}+\int_E\phi^2|\nabla u|^2e^{2h}e^{-f}+\frac{1}{2}\int_E\langle \nabla(\phi^2e^{2h}),\nabla u^2\rangle e^{-f}\notag\\
=&\int_E|\nabla(\phi e^h)|^2u^2e^{-f}+\int_E\phi^2|\nabla u|^2e^{2h}e^{-f}-\frac{1}{2}\int_E\phi^2\Delta_f(u^2)e^{2h}e^{-f}\notag\\
=&\int_E |\nabla(\phi e^h)|^2u^2e^{-f}-\int_E\phi^2u(\Delta_fu)e^{2h}e^{-f}\notag\\
\leq&\int_E|\nabla\phi|^2u^2e^{2h}e^{-f}+2\int_E\phi\langle\nabla\phi,\nabla h\rangle u^2e^{2h}e^{-f}\notag\\
&\int_E\phi^2|\nabla h|^2u^2 e^{2h}e^{-f}+\int_E\phi^2\mu u^2e^{2h}e^{-f}\label{21}
\end{align}
On the other hand, by assumption, we have
$$
\int_E\rho\phi^2e^{-f}\leq \int_E|\nabla\phi|^2e^{-f}-\int_E\mu\phi^2e^{-f},
$$
hence using \eqref{21}, we obtain
\begin{equation}
\label{22}\int_E\rho\phi^2u^2e^{2h}e^{-f}\leq \int_E|\nabla\phi|^2u^2e^{2h}e^{-f}+2\int_E\phi e^{2h}\langle\nabla\phi,\nabla h\rangle u^2e^{-f}+\int_E \phi^2|\nabla h|^2u^2e^{2h}e^{-f}.
\end{equation}
From now, we will devided the proof into three step.

\textbf{Step 1:} We claim that for any $0<\delta<1$, there exists a constant $0<C<\infty$ such that,
$$
\int_E\rho e^{2\delta r_\rho}u^2e^{-f}\leq C.
$$
Indeed, let us now choose
$$
\phi(r_\rho(x))=
\begin{cases}
r_\rho(x)-R_0 & \text{ on }E_\rho(R_0+1)\setminus E_\rho(R_0),\\
1&\text{ on } E_\rho(R)\setminus E_\rho(R_0+1),\\
R^{-1}(2R-r_\rho(x))&\text{ on }E_\rho(2R)\setminus E_\rho(R),\\
0&\text{ on }E\setminus E_\rho(2R).
\end{cases}
$$
It is easy to see that
$$
|\nabla\phi|^2(x)=
\begin{cases}
\rho(x)&\text{ on }E_\rho(R_0+1)\setminus E_\rho(R_0),\\
R^{-2}\rho(x)&\text{ on }E_\rho(2R)\setminus E_\rho(R),\\
0&\text{ on }(E_\rho(R)\setminus E_\rho(R_0+1))\cup (E\setminus E_\rho(2R)).
\end{cases}
$$
Moreover, we also choose
$$
h(r_\rho(x))=
\begin{cases}
\delta r_\rho(x)&\text{ for } r_\rho\leq\frac{K}{1+\delta},\\
K-r_\rho(x)&\text{ for } r_\rho\geq \frac{K}{1+\delta},
\end{cases}
$$
for some fixed $K>(R_0+1)(1+\delta)$. When $R\geq\frac{K}{1+\delta}$, we have
$$
|\nabla h|^2(x)=
\begin{cases}
\delta^2\rho(x)&\text{ for }r_\rho\leq \frac{K}{1+\delta},\\
\rho(x)&\text{ for }r_\rho\geq \frac{K}{1+\delta},
\end{cases}
$$
and
$$
\langle\nabla\phi,\nabla h\rangle(x)=
\begin{cases}
\delta\rho(x)&\text{ on } E_\rho(R_0+1)\setminus E_\rho(R_0),\\
R^{-1}\rho(x)&\text{ on }E_\rho(2R)\setminus E_\rho(R),\\
0&\text{ otherwise. }
\end{cases}
$$
Substituting into \eqref{22}, we infer
\begin{align}
&\int_E\rho\phi^2u^2e^{2h}e^{-f}\notag\\
\leq&\int_{E_\rho(R_0+1)\setminus E_\rho(R_0)}\rho u^2e^{2h}e^{-f}+R^{-2}\int_{E_\rho(2R)\setminus E_\rho(R)}\rho u^2e^{2h}e^{-f}\notag\\
&+2\delta\int_{E_\rho(R_0+1)\setminus E_\rho(R_0)}\rho u^2e^{2h}e^{-f}+2R^{-1}\int_{E_\rho(2R)\setminus E_\rho(R)}\rho u^2e^{2h}e^{-f}\notag\\
&+\delta^2\int_{E_\rho(\frac{K}{1+\delta})\setminus E_\rho(R_0)}\rho\phi^2u^2e^{2h}e^{-f}+\int_{E_{\rho}(2R)\setminus E_\rho(\frac{K}{1+\delta})}\rho\phi^2u^2e^{2h}e^{-f}.\notag
\end{align}
Therefore, by rearrangement the above inequality, we obtain
\begin{align*}
&\int_{E_\rho(\frac{K}{1+\delta})\setminus E_{\rho}(R_0+1)}\rho u^2e^{2h}e^{-f}\\
\leq&\int_{E_\rho(\frac{K}{1+\delta})}\rho\phi^2u^2e^{2h}e^{-f}\\
\leq&\int_{E_\rho(R_0+1)\setminus E_\rho(R_0)}\rho u^2e^{2h}e^{-f}+R^{-2}\int_{E_\rho(2R)\setminus E_\rho(R)}\rho u^2 e^{2h}e^{-f}\\
&+2\delta\int_{E_\rho(R_0+1)\setminus E_\rho(R_0)}\rho u^2e^{2h}e^{-f}+2R^{-1}\int_{E_\rho(2R)\setminus E_\rho(R)}\rho u^2e^{2h}e^{-f}\\&+\delta^2\int_{E_\rho(\frac{K}{1+\delta})\setminus E_\rho(R_0)}\rho u^2e^{2h}e^{-f}.
\end{align*}
Thus
\begin{align*}
&(1-\delta^2)\int_{E_\rho(\frac{K}{1+\delta})\setminus E_\rho(R_0+1)}\rho u^2e^{2h}e^{-f}\\
\leq&(1+\delta)^2\int_{E_\rho(R_0+1)\setminus E_\rho(R_0)}\rho u^2e^{2h}e^{-f}+R^{-2}\int_{E_\rho(2R)\setminus E_\rho(R)}\rho u^2e^{2h}e^{-f}\\&+2R^{-1}\int_{E_\rho(2R)\setminus E_\rho(R)}\rho u^2e^{2h}e^{-f}.
\end{align*}
Due to the definition of $h$ and the assumption on the growth condition of $u$, we see that the last two terms on the right hand side tend to as $R\to \infty$. Hence, we have the estimate
\begin{align*}
\int_{E_\rho(\frac{K}{1+\delta})\setminus E_\rho(R_0+1)}\rho u^2e^{2\delta r_\rho}e^{-f}\leq\frac{(1+\delta)^2}{1-\delta^2}\int_{E_\rho(R_0+1)\setminus E_\rho(R_0)}\rho u^2e^{2\delta r_\rho}e^{-f}.
\end{align*}
Note that the right hand side does not depend on $K$, by letting $K\to \infty$ we conclude that
\begin{equation}
\label{23}\int_{E\setminus E_\rho(R_0+1)}\rho u^2e^{2\delta r_\rho}e^{-f}\leq C,
\end{equation}
where $$
C=\frac{(1+\delta a)^2}{1-\delta^2}\int_{E(R_0+1)\setminus E(R_0)} u^2e^{2\delta ar}e^{-f}.
$$
\textbf{Step 2: }We want to prove that
\begin{equation}
\label{26}
\int_{E_\rho(R)}\rho u^2e^{2r_\rho}e^{-f}\leq C_6R.
\end{equation}
To do this, our first aim is to improve \eqref{23} by setting $h=r_\rho$ in the previous arguments. Now, suppose that $h=r_\rho$, by \eqref{22}, we infer
$$
-2\int_E\phi e^{2r_\rho}\langle\nabla\phi,\nabla r_\rho\rangle u^2e^{-f}\leq \int_E|\nabla \phi|^2u^2e^{2r_\rho}e^{-f}.
$$
For $R_0<R_1<R$, we choose
$$
\phi(x)=
\begin{cases}
\frac{r_\rho(x)-R_0}{R_1-R_0}&\text{ on }E_\rho(R_1)\setminus E_\rho(R_0),\\
\frac{R-r_\rho(x)}{R-R_1}&\text{ on }E_\rho(R)\setminus E_\rho(R_1).
\end{cases}
$$
Plugging $\phi$ in the above inequality, we obtain
\begin{align*}
&\frac{2}{R-R_1}\int_{E_\rho(R)\setminus E_\rho(R_1)}\left(\frac{R-r_\rho(x)}{R-R_1}\right)\rho u^2e^{2r_\rho}e^{-f}\\
\leq &\frac{1}{(R_1-R_0)^2}\int_{E_\rho(R_1)\setminus E_\rho(R_0)}\rho u^2e^{2 r_\rho}e^{-f}+\frac{1}{(R-R_1)^2}\int_{E_\rho(R)\setminus E_\rho(R_1)}\rho u^2e^{2r_\rho}e^{-f}\\&+\frac{2}{(R_1-R_0)^2}\int_{E_\rho(R_1)\setminus E_\rho(R_0)}(r_\rho-R_0)\rho u^2e^{2r_\rho}e^{-f}.
\end{align*}
Observe that for any $0<t<R-R_1,$ the following inequality
$$
\frac{2t}{(R-R_1)^2}\int_{E_\rho(R-t)\setminus E_\rho(R_1)}\rho u^2e^{2r_\rho}e^{-f}\leq \frac{2}{(R-R_1)^2}\int_{E_\rho(R)\setminus E_\rho(R_1)}(R-r_\rho)\rho u^2e^{2r_\rho}e^{-f},
$$
holds true. Therefore, we conclude that
\begin{align}
&\frac{2t}{(R-R_1)^2}\int_{E_\rho(R-t)\setminus E_\rho(R_1)}\rho u^2e^{2r_\rho}e^{-f}\notag\\
\leq&\left (\frac{2}{R_1-R_0}+\frac{1}{(R_1-R_0)^2}\right )\int_{E_\rho(R_1)\setminus E_\rho(R_0)}\rho u^2e^{2r_\rho}e^{-f}\notag\\&+\frac{1}{(R-R_1)^2}\int_{E_\rho(R)\setminus E_\rho(R_1)}\rho u^2e^{2r_\rho}e^{-f}.\label{24}
\end{align}
Now, by taking $R_1=R_0+1, t=1,$ and setting
$$
g(R)=\int_{E_\rho(R)\setminus E_\rho(R_0+1)}\rho u^2e^{2r_\rho}e^{-f},
$$
the inequality \eqref{24} becomes
$$
g(R-1)\leq C_1R^2+\frac{1}{2}g(R),
$$
where
$$
C_1=\frac{3}{2}\int_{E_\rho(R_0+1)\setminus E_\rho(R_0)}\rho u^2e^{2r_\rho}e^{-f}
$$
is independent of $R$. Iterating this inequality, we obtain that for any positive integer $k$ and $R\geq 1$
\begin{align*}
g(R)\leq&C_1\sum_{i=1}^k\frac{(R+i)^2}{2^{i-1}}+2^{-k}g(R+k)\\
\leq &C_1R^2\sum_{i=1}^{\infty}\frac{(1+i)^2}{2^{i-1}}+2^{-k}g(R+k)\\
\leq&C_2R^2+2^{-k}g(R+k)
\end{align*}
for some constant $C_2$. Note that in our previous estimate \eqref{23}, we have proved the following inequality
$$
\int_E\rho u^2e^{2\delta r_\rho}e^{-f}\leq C
$$
for any $\delta<1$. Thus, this implies that
\begin{align*}
g(R+k)=&\int_{E_\rho(R+k)\setminus E_\rho(R_0+1)}\rho u^2e^{2r_\rho}e^{-f}\\
\leq&e^{2(R+k)(1-\delta)}\int_{E_\rho(R+k)\setminus  E_\rho(R_0+1)}\rho u^2e^{2\delta r_\rho}e^{-f}\\
\leq&Ce^{2(R+k)(1-\delta)}.
\end{align*}
Now, if we choose $2(1-\delta)<\ln 2$ then 
$$
2^{-k}g(R+k)\to 0
$$
as $k\to\infty$. Consequently, we obtain
$$
g(R)\leq C_2R^2.
$$
Therefore, by adjusting the constant, we conclude
\begin{equation}
\label{25}\int_{E_\rho(R)}\rho u^2e^{2r_\rho}e^{-f}\leq C_3R^2
\end{equation}
for all $R\geq R_0$.

Finally, using inequality \eqref{24} again and by choosing $R_1=R_0+1$ and $t=R/2$ this times, we infer
$$
R\int_{E_\rho(\frac{R}{2})\setminus E_\rho(R_0+1)}\rho u^2e^{2r_\rho}e^{-f}\leq C_4 R^2+\int_{E_\rho(R)\setminus E_\rho(R_0+1)}\rho u^2e^{2r_\rho}e^{-f}.
$$
Observe that the second term on the right hand side is bounded by \eqref{25}, we have
$$
\int_{E_{\rho}(\frac{R}{2})\setminus E_\rho(R_0+1)}\rho u^2e^{2r_\rho}e^{-f}\leq C_5 R.
$$
Therefore, for $R\geq R_0$, the claim is prove.

\textbf{Step 3: }In this step, we will complete the proof of Lemma \ref{1} by using \eqref{26}
$$
\int_{E_{\rho}(R)\setminus E_\rho(R_0+1)}\rho u^2e^{2r_\rho}e^{-f}\leq C_6R.
$$
Indeed, letting $t=2$ and $R_1=R-4$ in \eqref{24}, we obtain
\begin{align*}
&\int_{E_\rho(R-2)\setminus E_\rho(R-4)}\rho u^2e^{2r_\rho}e^{-f}\\
\leq &\left (\frac{8}{R-R_0-4}+\frac{4}{(R-R_0-4)^2}\right )\int_{E_\rho(R-4)\setminus E_\rho(R_0)}\rho u^2e^{2r_\rho} e^{-f}\\&+\frac{1}{4}\int_{E_\rho(R)\setminus E_\rho(R-4)}\rho u^2e^{2r_\rho}e^{-f}.
\end{align*}
Thanks to \eqref{26}, the first term of the right hand side is estimated by
$$
C(1+(R-R_0-4)^{-1})
$$
for $R-4\geq 2(R_0+1)$. Hence, by renaming $R$, the above inequality can be rewritten as
$$
\int_{E_\rho(R-2)\setminus E_\rho(R-4)}\rho u^2e^{2r_\rho}e^{-f}\leq C(1+(R-R_0)^{-1})+\frac{1}{3}\int_{E_\rho(R)\setminus E_\rho(R-2)}\rho u^2e^{2r_\rho}e^{-f}.
$$
Iterating this inequality $k$ times, we conclude that
$$
\int_{E_\rho(R+2)\setminus E_\rho(R)}\rho u^2e^{2r_\rho}e^{-f}\leq C(1+(R-R_0)^{-1})\sum_{i=0}^{k-1}3^{-i}+3^{-k}\int_{E_\rho(R+2(k+1))\setminus E_\rho(R+2k)}\rho u^2e^{2r_\rho}e^{-f}.
$$
However, using \eqref{26} again, we deduce that the second term is bounded by
$$
3^{-k}\int_{E_\rho(R+2(k+1))\setminus E_\rho(R+2k)}\rho u^2e^{2r_\rho}e^{-f}\leq C_63^{-k}(R+2(k+1))
$$
which tends to $0$ as $k\to\infty$. This implies
$$
\int_{E_\rho(R+2)\setminus E_\rho(R)}\rho u^2e^{2r_\rho}e^{-f}\leq C(1+(R-R_0)^{-1})
$$
for some constant $C>0$ independent of $R$, and the lemma follows.
\end{proof}
\begin{remark}Recently, in \cite{MSW}, Munteanu et. all  introduced a parabolic version of decay estimate for weighted volume and used it to investigate Poisson equation on complete smooth metric measure spaces.
\end{remark}
\begin{corollary}\label{coro1}
Let $E$ be an end of a complete smooth metric measure space $(M,ds^2_\rho,e^{-f}dv)$. Suppose that $\lambda_{1, f}(E)>0$, i. e.,
$$
\lambda_{1, f}\int_E\varphi^2 e^{-f}\leq\int_E|\nabla\varphi|^2e^{-f}
$$
for any compactly supported function $\varphi \in \mathcal{C}_0^\infty(E)$. Let $u$ be a nonnegative function defined on $E$ such that
$$
(\Delta_f+\mu)u\geq 0
$$
for some constant $\mu$ satisfying $\lambda_{1, f}-\mu>0$. If $u$ has the growth condition
$$
\int_{E(R)} u^2e^{-2ar}e^{-f}=o(R)
$$
when $R\to\infty$, where $a=\sqrt[]{\lambda_{1, f}-\mu}$, then $u$ must satisfy the decay estimate
$$
\int_{E(R+1)\setminus E(R)} u^2e^{-f}\leq C(1+(R-R_0)^{-1})e^{-2aR}\int_{E(R_0+1)\setminus E(R_0)}e^{2ar}u^2e^{-f}
$$
for some constant $C>0$ depending on $f, $ and $a$.
\end{corollary}
\begin{proof}[Proof of Corollary \ref{coro1}]
By variational principle for $\lambda_{1, f}$, we have
$$
(\lambda_{1, f}-\mu)\int_E\varphi^2e^{-f}\leq\int_E|\varphi|^2e^{-f}-\int_E\mu\varphi^2e^{-f}.
$$
Let $\rho=a^2$, the distance function with respect to the complete metric $\rho ds^2$ is given by
$$
r_\rho(x)=ar(x).
$$
Now, we can apply Lemma \ref{1} to complete the proof.
\end{proof}
Note that if $u\equiv 1$, we obtain the following decay estimate
\begin{corollary}\label{mrl12}
Suppose that $E$ is a $f$-parabolic end of smooth metric measure space $(M, g, e^{-f}dv)$ with $\lambda_{1, f}(E)>0$. Denote by $V_{E, f}(R)$ the weighted volume of $E(R)$, then the following decay estimate
$$\begin{aligned}
V_{E, f}(R+1)&-V_{E, f}(R)\\
&\leq C(\lambda_{1, f}(E))\left(1+\frac{1}{R-R_0}\right)e^{2\ \sqrt[]{\lambda_{1, f}}(R-R_0)}(V_{E, f}(R_0+1)-V_{E, f}(R_0))
\end{aligned}$$
holds true. Here $R_0$ is a given positive number and $C(\lambda_{1, f}(E))$ is some positive constant depending on $\lambda_{1, f}(E)$.
\end{corollary}
It is worth to mention that in \cite{BK}, Buckley and Koskela also proved earlier a version of decay estimate for $f$-volume in a more general setting.
\section{Counting cusps on smooth metric measure spaces of finite $f$-volume}\label{section3}
As what we mentioned in the introduction part, in order to estimate the number of cusps we must have a decay estimate of the $f$-volume proved in Section \ref{section2} and a volume comparison theorem. Hence, first we introduce a volume comparison result given by Munteanu and Wang in \cite{MW}.
\begin{lemma}\label{vc}
Let $(M,g,e^{-f}dv)$ be a complete smooth metric measure space with $Ric_f\geq -(n-1)$. Suppose the weighted function has sublinear growth, i. e. for $x\in M$,
$$ |f|(x)\leq \alpha r(x)+\beta,$$
for some nonnegative constants $\alpha$ and $\beta$. Then there exists a constant $C>0$ such that the volume upper bound
$$ V_f(B(o,R))\leq Ce^{(n-1+\alpha)R}V_f(B(o,1))$$
holds for all $R>0$. Here $V_f(B(o, R))$ stands for the weighted volume of $B(o, R)$
\end{lemma}
Now, we will combine Lemma \ref{vc} and Corollary \ref{mrl12} to count number of cusps. The first result is as follows.
\begin{theorem}\label{main2}
Let $(M, g, e^{-f}dv)$ be a smooth metric measure space with $Ric_f\geq -(n-1)$. Suppose that the weighted function is of sublinear growth, namely, for all $x\in M$,
$$ |f|(x)\leq \alpha r(x)+\beta,$$
for some nonnegative constants $\alpha$ and $\beta$. If $M$ has finite $f$-volume given by $V$, and
$$ \lambda_{1, f}(M\setminus B_o(R_0))\geq \frac{(n-1+\alpha)^2}{4},$$
for some $R_0>0$, then
$$ N(M)\leq C e^{(n-1+\alpha)R_0}\frac{V}{V_{f}(B(o,1))}.$$
Here $C>0$ is a constant depending on $\lambda_{1, f}$.
\end{theorem}
\begin{proof}[Proof of Theorem \ref{main2}]To simple the notation, we denote $V_{o, f}(R)$ to be the weighted volume of the geodesic ball $B(o,R)$, then for all $R>2(R_0+1)$, we have
\begin{equation}\label{241-31}
\begin{aligned}
V_{o, f}(R+2)&-V_{o, f}(R)\\
&\leq C\left(1+\frac{1}{R-R_0}\right)e^{(n-1+\alpha)(R_0-R)}(V_{o, f}(R_0+1)-V_{o, f}(R_0)).
\end{aligned}
\end{equation}
Here we used Lemma \ref{mrl12}.

On the other hand, if $y\in\partial B(o,R+1)$ then $B(o, 1)\subset B(y, R+2)$. Hence, we use Lemma \ref{vc} to obtain
\begin{equation}\label{241-32}
V_{o, f}(1)\leq V_{y, f}(R+2)\leq C_1e^{(n-1+\alpha)R}V_{y, f}(1)
\end{equation}
Suppose that $M\setminus B(o,R)$ has $N(R)$ unbounded components, then there exist $N(R)$ number of points $\{y_i\in \partial B(o,R+1)\}$ such that $B({y_i},1)\cap B({y_j},1)=\oslash$ for $i\neq j$. In particular, applying \eqref{241-32} to each of the $y_i$ and combining with \eqref{241-31}, we have
\begin{align*}
N(R)C_1^{-1}e^{-(n-1+\alpha)R}V_{o, f}(1)&\leq \sum_{i=1}^{N(R)}V_{y_i, f}(1)\\
&\leq V_{o, f}(R+2)-V_{o, f}(R)\\
&\leq C\left(1+\frac{1}{R-R_0}\right)e^{(n-1+\alpha)(R_0-R)}(V_{o, f}(R_0+1)-V_{o, f}(R_0)).
\end{align*}
This implies that
$$
N(R)\leq C C_1(1+(R-R_0)^{-1})e^{(n-1+\alpha)R_0}(V_{o,f}(R_0+1)-V_{o,f}(R_0))V_{o,f}^{-1}(1).
$$
Note that $N(R)$ is the number of ends of $M$ with respect to $B(o, R)$, letting $R\to \infty$, we complete the proof.
\end{proof}
Next, we will derive a weighted version of a Li-Wang's result in \cite{LW2010} to estimate $\lambda_{1, f}(B_o(R))$ of a geodesic ball centered at $o$ with radius $R$ in terms of the weighted volume of the ball. It is worth to mention that as in \cite{LW2010}, we do not require any curvature assumptions on $M$.
\begin{lemma}\label{lm2.2}
Let $(M,g,e^fdv)$ be a complete smooth metric measure space. Then for any $0<\delta<1, R>2$ and $o\in M$, we have
$$\lambda_{1,f}(B_o(R))\leq\dfrac{1}{4\delta^2(R-1)^2}\Big(\ln \Big(\dfrac{V_{o,f}(R)}{V_{o,f}(1)}\Big)+\ln \Big(\dfrac{81}{1-\delta}\Big)\Big)^2$$
\end{lemma}
\begin{proof}Observe that
$$\dfrac{4}{R^2}\leq \dfrac{1}{4\delta^2(R-1)^2}\ln^2\Big(\dfrac{81}{1-\delta}\Big)$$
Hence, we may assume $\lambda_{1, f}(B_o(R))\geq\dfrac{4}{R^2}$.

To simple the notation, let us use $\lambda_{1,f}$ to denote $\lambda_{1,f}(B_o(R))$. By the variational characteristic of $\lambda_{1,f}(B_o(R))$, we have
 \begin{align*}
 \lambda_{1,f}&\int_M\phi^2\exp (-2\delta\sqrt{\lambda_{1, f}}r)e^{-f}\leq \int_M |\nabla(\phi \exp (-\delta\sqrt{\lambda_{1, f}}r))|^2e^{-f}\notag\\&
=\int_{M}|\nabla \phi|^2\exp (-2\delta \sqrt{\lambda_{1, f}}r)e^{-f}-2\delta
\sqrt{\lambda_{1, f}}\int_M \phi \exp (-2\delta \sqrt{\lambda_{1, f}}r)\langle \nabla\phi,\nabla r\rangle e^{-f}\notag\\&
+\delta^2 \lambda_{1, f} \int_{ M} \phi^2 \exp (-2\delta\sqrt{\lambda_{1, f}}r)e^{-f},
\end{align*}
for any nonnegative Lipschitz function $\phi$ with support in $B_o(R)$. Consequently,
\begin{align*}
\lambda_{1, f}&(1-\delta^2)\int_M\phi^2\exp (-2\delta\sqrt{\lambda_{1, f}}r)e^{-f}\notag\\&\leq \int_{M}|\nabla \phi|^2\exp (-2\delta \sqrt{\lambda_{1, f}}r)e^{-f}-2\delta
\sqrt{\lambda_{1, f}}\int_M \phi \exp (-2\delta \sqrt{\lambda_{1, f}}r)\langle \nabla\phi,\nabla r\rangle e^{-f}.
\end{align*}
In particular, for $R>2,$ we choose
$$\phi=\begin{cases}1 \ \ \mbox{on} \ \ B_o(R-\lambda_{1, f}^{1/2})\\
\sqrt{\lambda_{1, f}}(R-r) \ \ \mbox{on} \ \ B_o(R)\backslash B_o(R-\lambda_{1, f}^{1/2})\\
0\ \ \mbox{on} \ \ M\backslash B_o(R)
\end{cases}$$
then $\phi =1$ on $B_o(1)$ since $R-\lambda_{1,f}^{1/2}\geq R/2>1$. Plugging $\phi$ in the above inequality, we obtain
\begin{align*}
(1-\delta^2)&\lambda_{1, f} \exp (-2\delta\sqrt{\lambda_{1, f}})V_{p,f}(1)\\
&\leq(1-\delta^2)\lambda_{1,f} \int_{M}\phi^2\exp (-2\delta\sqrt{\lambda_{1, f}}r)e^{-f}\\
&=\int_{M}|\nabla\phi|^2e^{-2\delta\sqrt{\lambda_{1,
f}}r}e^{-f}-2\delta\sqrt{\lambda_{1, f}}\int_{M}\phi e^{
-2\delta\sqrt{\lambda_{1, f}}r}\langle \nabla\phi,\nabla r\rangle
e^{-f}\\
& \leq \lambda_{1, f}\int\limits_{B_o(R)\setminus
B_o(R-\lambda_{1,f}^{1/2})}e^{
-2\delta\sqrt{\lambda_{1, f}}r}e^{-f}+2\delta\lambda_{1,
f}\int\limits_{B_o(R)\setminus B_o(R-\lambda_{1,f}^{1/2})}\phi e^{
-2\delta\sqrt{\lambda_{1, f}}r} e^{-f}\\
&\leq (1+2\delta)
\lambda_{1, f}e^{-2\delta(\sqrt{\lambda_{1, f}}R-1)}V_{o,f}(R).
\end{align*}
Here we used $\sqrt[]{\lambda_{1, f}}(R-r)\leq 1$ on $B_o(R)\setminus B_o(R-\lambda_{1,f}^{1/2})$ in the last inequality. Therefore,
$$e^{2\delta\sqrt{\lambda_{1, f}}(R-1)}\leq\dfrac{(1+2\delta)e^{2\delta}}{1-\delta^2}\dfrac{V_f(B_o(R))}{V_f(B_o(1))}\leq\dfrac{27}{1-\delta}\dfrac{V_f(B_o(R))}{V_f(B_o(1))}.$$
This implies
$$2\delta\sqrt{\lambda_{1, f}}(R-1)\leq\ln \Big(\dfrac{27}{1-\delta}\Big)+\ln\Big(\dfrac{V_f(B_o(R))}{V_f(B_o(1))}\Big)$$
The lemma follows by rewriting this inequality.
\end{proof}
Note that $\lambda_{1, f}(M)=\lim_{R\rightarrow\infty}\lambda_{1, f}(B_o(R))$. If we first let $R$ going to infinity and then $\delta$ going to $1$ in the estimate of Lemma \ref{lm2.2}, we have the following result.
\begin{corollary}Let $(M,g,e^fdv)$ be a complete smooth metric measure space. Assume for some nonnegative constants $\alpha$ and $\beta$, $$ |f|(x)\leq \alpha r(x)+\beta$$
for $x\in M$ and $\lambda_{1, f}(M)$ its bottom spectrum. Then
$$\lambda_{1, f}(M)\leq\dfrac{1}{4}\Big(\liminf_{R\rightarrow\infty}\dfrac{\ln V_f(B_o(R))}{R}\Big)^2$$
\end{corollary}
Now, we will give a proof of Theorem \ref{main3}.
\begin{proof}[Proof of Theorem \ref{main3}]
Note that $o\in M$ be a fixed point. For any $0<\delta<1,$ let $$R_0=\dfrac{1}{(n-1+\alpha)\delta}\Big(\ln \big(\dfrac{81}{1-\delta}\Big)+\ln \Big(\dfrac{V}{V_{o,f}(1)}\Big)\Big)+3.$$
Thanks to Lemma \ref{lm2.2}, we have
\begin{equation}\label{pt1}
\lambda_{1, f}(B_o(R_0))\leq\dfrac{\delta^2(n-1+\alpha)^2}{4\delta^2}\dfrac{\left(\ln \Big(\dfrac{V_{o,f}(R)}{V_{o,f}(1)}\Big)+\ln \Big(\dfrac{81}{1-\delta}\Big)\right)^2}{\left(\ln\Big(\dfrac{81}{1-\delta}\Big)+\ln \Big(\dfrac{V}{V_{o,f}(1)}\Big)\right)^2}\leq \dfrac{(n-1+\alpha)^2}{4}
\end{equation}
On the other hand, by the variational principle, we have
$$\mu_1(M)\leq\max \{\lambda_{1, f}(B_o(R_0)), \lambda_{1, f}(M\backslash B_o(R_0))\}.$$
Hence, combining this inequality with assumption regarding to $\mu_1(M)$, we infer
$$\lambda_{1, f}(M\backslash B_o(R_0))\geq \dfrac{(n-1+\alpha)^2}{4}.$$
So Theorem \ref{main2} implies
$$N(M)\leq C(n)V_fV^{-1}_f(B_o(1))\exp ((n-1+\alpha)R_0).$$
To finish the proof, we first choose
$$\delta=1-\dfrac{1}{\ln (V_fV^{-1}_f(B_o(1)))}$$
then replace the value of $R_0, \delta$ in the last inequality. So we are done.
\end{proof}
As we mentioned in the introduction part, two main ingredients in our proof of Theorem \ref{main3} are the decay estimate of the volume and the volume comparison result. Therefore, using Lemma \ref{mrl12} and the volume comparison theorems in \cite{MW14} and \cite{Lam}, we can count the number of cusps on the following two manifolds.
\begin{theorem}Let $(M,g,e^fdv)$ be a complete smooth metric measure space with $Ric_f\geq\lambda$. Assume for some nonnegative constants $\alpha$ and $\beta$,
$$ |f|(x)\leq \alpha r(x)+\beta$$
for all $x\in M$. If $M$ has finite $f$-volume given by $V$, and
$$\mu_1(M)\geq\dfrac{(n-1)(\alpha-\lambda)}{4}$$
Let us denote $N(M)$ to be the number of ends (cusps) of $M$. Then there exists a contant $C(n)>0$ depending only on $n$, such that,
$$N(M)\leq C(n)\Big(\dfrac{V}{V_{o}(1)}\Big)^2\ln\Big(\dfrac{V}{V_{o}(1)}\Big)$$
where $V_{o}(1)$ denotes the (non-weighted) volume of the unit ball centered at any fixed point $o\in M$. Here $\mu_1(M)$ is defined by the (non-weighted) Reileigh quotient,
$$\mu_1(M)=\inf\limits_{\phi\in H^1(M), \int_M\phi=0}\frac{\int_M|\nabla\phi|^2}{\int_M\phi^2}.$$
\end{theorem}
\begin{theorem}Let $M$ be a complete noncompact $16$-dimensional manifold with holonomy group Spin(9) with $\mu_1(M)\geq 121$. Then there exists a contant $C(n)>0$ depending only on $n$, such that,
$$N(M)\leq C(n)\Big(\dfrac{V}{V_{o}(1)}\Big)^2\ln\Big(\dfrac{V}{V_{o}(1)}\Big)$$
where $V_{o}(1)$ denotes the (non-weighted) volume of the unit ball centered at any point $o\in M$. Here $\mu_1(M)$ is defined by the (non-weighted) Reileigh quotient,
$$\mu_1(M)=\inf\limits_{\phi\in H^1(M), \int_M\phi=0}\frac{\int_M|\nabla\phi|^2}{\int_M\phi^2}.$$
\end{theorem}
\section{Counting cusps via the $p$-Laplacian}\label{section4}
In this section, we will use the nonlinear theory of $p$-Laplacian to estimate the number of cusps of Riemannian manifolds.
Again, our strategy is to use a (nonlinear) decay estimate of the volume and corresponding volume comparison theorem.
Therefore, let us recall the following nonlinear version regarding to the rates of volume decay.
\begin{theorem}[\cite{BK}]\label{BKp}
Let $E$ be an end of a complete Riemannian manifold $(M^n, g)$ with
respect to $\overline{B}(o, R_0)$. Suppose that the first
eigenvalues of the $p$-Laplacian $\lambda_{1,p}(E)>0, (1\leq
p<\infty)$. If $E$ has finite volume then the following decay
estimate
$$
V_{E}(R+1)-V_{E}(R)\leq Ce^{-p\lambda_{1,p}^{1/p}R}
$$
for all $R> R_0+2$.
\end{theorem}
\begin{proof}We note that in the proof of the first part of Theorem 0.1 in \cite{BK},
the author showed that for $R>R_0+2$, we have
$$V_E(R+1)-V_E(R)\leq Ce^{-p\lambda_{1,p}^{1/p}R}.$$
The proof is complete.
\end{proof}
We can estimate the number of cusps of smooth metric measure spaces as follows.
\begin{theorem}\label{4main1}
Let $(M^n, g)$ be a smooth metric measure space with $Ric\geq
-(n-1)$. If $M$ has finite volume given by $V$, and
$$ \lambda_{1, p}(M\setminus B_o(R_0))\geq \frac{(n-1)^p}{p^p},$$
for some $R_0>0$, then
$$ N(M)\leq C e^{(n-1)R_0}\frac{V}{V_0(1)}.$$
Here $C>0$ is a constant depending on $\lambda_{1, p}$.
\end{theorem}
\begin{proof}By using the decay estimate in Theorem \ref{BKp} and the volume comparison theorem \ref{vc}, we can repeat the proof of Theorem \ref{main2} to derive the conclusion. Since they are almost the same, we omit the detail.
\end{proof}
Next, we will estimate $\lambda_{1, p}$ on the ball $B(o, R)$.
\begin{lemma}\label{4lm2.2}
Let $(M^n,g)$ be a Riemannian manifold. Then for any $0<\delta< 1,
1\leq p<\infty$ and $o\in M$, we have
$$\lambda_{1,p}(B_o(R))\leq\dfrac{1}{\delta^p(R-1)^p}\Big(2+\frac{\alpha(p)}{p}\ln\dfrac{1}{1-\delta}+\dfrac{1}{p}\ln\dfrac{V_0(R)}{V_0(1)}\Big)^p$$
where $\alpha(p)=\max\{1, p/2\}$.
\end{lemma}
\begin{proof}Observe that
$$\dfrac{2^p}{R^p}\leq \dfrac{1}{\delta^p(R-1)^p}\Big(2+\dfrac{\alpha(p)}{p}\ln\dfrac{1}{1-\delta}\Big)^p.$$

Hence, we may assume $\lambda_{1, p}(B_o(R))\geq\dfrac{2^p}{R^p}$.

To simple the notation, let us use $\lambda_{1,p}$ to denote $\lambda_{1,p}(B_o(R))$.
By the variational characteristic of $\lambda_{1,p}(B_o(R))$, we have
 \begin{align}
 \lambda_{1,p}&\int_M\phi^p\exp (-p\delta\sqrt[p]{\lambda_{1,
 p}}r)\notag\\&
\leq \int_M |\nabla(\phi \exp (-\delta\sqrt[p]{\lambda_{1,
p}}r))|^p\notag\\&
=\int_{M}e^{-p\delta\sqrt[p]{\lambda_{1,p}}r}|\nabla\phi-\delta\sqrt[p]{\lambda_{1,p}}\phi\nabla
r|^p\notag\\& =
\int_{M}e^{-p\delta\sqrt[p]{\lambda_{1,p}}r}\Big(|\nabla\phi|^2+2\delta\sqrt[p]{\lambda_{1,p}}|\nabla\phi||\nabla
r| +\delta^2 \lambda_{1,p}^{2/p}  \phi^2 \Big)^{p/2}, \label{case}
\end{align}
for any nonnegative Lipschitz function $\phi$ with support in
$B_o(R)$. We have two cases.

\noindent \textbf{Case 1: }$p\geq 2$. Observe that $x^{p/2}$ is a
convex function, we have the following basic inequality
$$(A+B)^{p/2}=\Big((1-\delta^2)\dfrac{A}{1-\delta^2}+\delta^2\dfrac{B}{\delta^2}\Big)^{p/2}
\leq (1-\delta^2)^{1-p/2}A^{p/2}+\delta^{2-p}B^{p/2},$$ for any $A,
B\geq0$. Hence, by \eqref{case}, we have
$$\begin{aligned}
\lambda_{1,p}\int_M\phi^p\exp (-p\delta\sqrt[p]{\lambda_{1,
 p}}r)
 \leq& (1-\delta^2)^{1-p/2}\int_{M}e^{-p\delta\sqrt[p]{\lambda_{1,p}}r}
 \Big(|\nabla\phi|^2 +2\phi\delta\sqrt[p]{\lambda_{1,p}}|\nabla\phi||\nabla
r|\Big)^{p/2}\\
&+\delta^{2}\lambda_{1,p}\int_Me^{-p\delta\sqrt[p]{\lambda_{1,p}}r}\phi^p.
\end{aligned}$$
Consequently,
\begin{align*}
\lambda_{1,p}\int_M\phi^p\exp (-p\delta\sqrt[p]{\lambda_{1,
 p}}r)
& \leq \dfrac{1}{(1-\delta^2)^{p/2}}\int_{M}e^{-p\delta\sqrt[p]{\lambda_{1,p}}r}
 \Big(|\nabla\phi|^2 +2\phi\delta\sqrt[p]{\lambda_{1,p}}|\nabla\phi||\nabla
r|\Big)^{p/2}\\
& \leq \dfrac{1}{(1-\delta)^{p/2}}\int_{M}e^{-p\delta\sqrt[p]{\lambda_{1,p}}r}
 \Big(|\nabla\phi|^2 +2\phi\delta\sqrt[p]{\lambda_{1,p}}|\nabla\phi||\nabla
r|\Big)^{p/2}.
\end{align*}
Here we used $0<\delta<1$ in the last inequality.

\noindent
\textbf{Case 2: }$1\leq p<2$. Since $0<p/2<1$, we have for
$A,B\geq0$
$$(A+B)^{p/2}\leq A^{p/2}+B^{p/2}.$$
Therefore, by \eqref{case}, we obtain
$$\begin{aligned}
\lambda_{1,p}\int_M\phi^p\exp (-p\delta\sqrt[p]{\lambda_{1,
 p}}r)
 \leq& \int_{M}e^{-p\delta\sqrt[p]{\lambda_{1,p}}r}
 \Big(|\nabla\phi|^2 +2\phi\delta\sqrt[p]{\lambda_{1,p}}|\nabla\phi||\nabla
r|\Big)^{p/2}\\
&+\delta^{2}\lambda_{1,p}\int_Me^{-p\delta\sqrt[p]{\lambda_{1,p}}r}\phi^p.
\end{aligned}$$
Since $0<\delta<1$, this implies,
\begin{align*}
\lambda_{1,p}\int_M\phi^p\exp (-p\delta\sqrt[p]{\lambda_{1,
 p}}r)
 \leq \dfrac{1}{1-\delta}\int_{M}e^{-p\delta\sqrt[p]{\lambda_{1,p}}r}
 \Big(|\nabla\phi|^2 +2\phi\delta\sqrt[p]{\lambda_{1,p}}|\nabla\phi||\nabla
r|\Big)^{p/2}.
\end{align*}
In conclusion, we obtain in both cases that
$$\lambda_{1,p}\int_M\phi^p\exp (-p\delta\sqrt[p]{\lambda_{1,
 p}}r)
 \leq \dfrac{1}{(1-\delta)^{\alpha(p)}}\int_{M}e^{-p\delta\sqrt[p]{\lambda_{1,p}}r}
 \Big(|\nabla\phi|^2 +2\phi\delta\sqrt[p]{\lambda_{1,p}}|\nabla\phi||\nabla
r|\Big)^{p/2}.
$$
Now,
for $R>2,$ we choose
$$\phi=\begin{cases}1 \ \ \mbox{on} \ \ B_0(R-\lambda_{1, f}^{-1/p})\\
\sqrt[p]{\lambda_{1, p}}(R-r) \ \ \mbox{on} \ \ B_o(R)\backslash B_0(R-\lambda_{1, f}^{-1/p})\\
0\ \ \mbox{on} \ \ M\backslash B_0(R)
\end{cases}$$
then $\phi =1$ on $B_0(1)$ since $R-\lambda_{1,p}^{-1/p}\geq
\frac{R}{2}>1$. Plugging $\phi$ in the above inequality, we obtain

\begin{align*}
&\lambda_{1,p} \exp (-p\delta\lambda_{1,p}^{1/p})V_o(1))
\\&\leq\lambda_{1,p} \int_{M}\phi^p\exp (-p\delta\sqrt[p]{\lambda_{1, f}}r)\\&
\leq \dfrac{\lambda_{1,p}}{(1-\delta)^{\alpha(p)}}\int_{B_0(R)\backslash
B_0(R-\lambda_{1,p}^{1/p})}(1+2\phi\delta)^{p/2}e^{
-p\delta\sqrt[p]{\lambda_{1,p}}r}\\&
\leq\dfrac{\lambda_{1,p}}{(1-\delta)^{\alpha(p)}}3^{p/2}e^{-p\delta(\lambda_{1,p}^{1/p}R-1)}V_0(R).
\end{align*}
Here we used $0\leq\phi\leq1$ in the third inequality.
Therefore,
$$e^{p\delta\lambda_{1, p}^{1/p}(R-1)}\leq \dfrac{3^{p/2}e^{p\delta}}{(1-\delta)^{\alpha(p)}}\dfrac{V_o(R)}{V_{o}(1)}
\leq\dfrac{2^pe^p}{(1-\delta)^{\alpha(p)}}\dfrac{V_0(R)}{V_0(1)}.$$ This
implies
$$p\delta\lambda_{1, p}^{1/p}(R-1)\leq p(\ln2+1)+\alpha(p)\ln\dfrac{1}{1-\delta}+ \ln\Big(\dfrac{V_o(R)}{V_o(1)}\Big).$$
The lemma follows by rewriting this inequality.
\end{proof}
\begin{corollary}
Let $M$ be a complete smooth metric measure spaces and $\lambda_{1, p}$ is its weighted $p$-spectrum. Then
$$\lambda_{1, p}\leq \left(\dfrac{1}{p}\liminf\limits_{R\to\infty}\frac{\ln V_{o, f}(R)}{R}\right)^p.$$
\end{corollary}
\begin{proof}We will use Lemma \ref{4lm2.2} to give the proof. First, we let
$R\to\infty$ then let $\delta\to1$, we obtain the conclusion. Hence,
the proof is complete.
\end{proof}
Now we give a proof of Theorem \ref{12main12}.
\begin{proof}[Proof of Theorem \ref{12main12}]
Note that $o\in M$ be a fixed point. Let
$$R_0=\dfrac{p}{(n-1)\delta}\Big(2+\frac{\alpha(p)}{p}\ln\dfrac{1}{1-\delta}+\frac{1}{p}\ln
\frac{V}{V_{o}(1)}\Big)+3.$$
Thanks to Lemma \ref{4lm2.2}, we infer
\begin{equation*}
\lambda_{1, p}(B_o(R_0))\leq \dfrac{(n-1)^p}{p^p}
\end{equation*}
On the other hand, by the variational principle, we have
$$\mu_{1,p}(M)\leq\max \{\lambda_{1, p}(B_o(R_0)), \lambda_{1, p}(M\backslash B_o(R_0))\}.$$
Hence, combining this inequality with assumption regarding to
$\mu_{1,p}(M)$, we obtain
$$\lambda_{1, f}(M\backslash B_o(R_0))\geq \dfrac{(n-1)^p}{p^p}.$$
So Theorem \ref{4main1} implies
$$ N(M)\leq C e^{(n-1)R_0}\frac{V}{V_{o}(1)}.$$
To finish the proof, we first choose $\delta =1-\dfrac{1}{\ln
(V.V_{o}(1)^{-1})}$ then replace the value of $R_0, \delta$ in the
last inequality. So we are done.
\end{proof}
Similarly, using the volume comparison theorems in \cite{Lam, KLZ,
LW05}, we have the following theorems.
\begin{theorem} Let $M^m$ be a complete K\"{a}hler manifold of complex dimension $m$ with finite volume. Assume that
$M$ has holomorphic bisectional curvatures satisfying
$$R_{i\bar{j}\bar{i}\bar{j}}\geq-(1+\delta_{ij})$$ for all
unitary frame $\{e_1, \ldots, e_m\}$. If
$$\mu_{1,p}(M)\geq\left(\dfrac{2m}{p}\right)^p,$$
then  there exists a contant $C(n,p)>0$ depending only
on $n$ and $p$ such that,
$$N(M)\leq C(n,p)\Big(\frac{V}{V_{o}(1)}\Big)^2\ln^{\alpha(p)} \Big(\frac{V}{V_{o}(1)}\Big)$$
where $\alpha(p)=\max\{1, p/2\}$ and $V_o(1)$ denotes the volume of the unit ball centered at any
point $o\in M$
\end{theorem}
\begin{theorem}
Let $(M^n,g)$ be a complete quarternionic K\"{a}hler of real
dimension $4m$ with finite volume given by $V$. Assume that its
scalar curvature of satisfies the bound
$$S_M\geq -16m(m+2)$$
and
 $$\mu_{1,p}(M)\geq\dfrac{(2(2m+1))^p}{p^p}.$$
 Let us denote
$N(M)$ to be the number of ends (cusps) of $M$. Then  there exists a contant $C(n,p)>0$ depending only
on $n$ and $p$ such that,
$$N(M)\leq C(n,p)\Big(\frac{V}{V_{o}(1)}\Big)^2\ln^{\alpha(p)} \Big(\frac{V}{V_{o}(1)}\Big)$$
where $\alpha(p)=\max\{1, p/2\}$ and $V_o(1)$ denotes the volume of the unit ball centered at a
fixed point $o\in M$.
\end{theorem}
\begin{theorem}
Let $(M^n,g)$ be a Let $M$ be a complete noncompact $16$-dimensional
manifold with holonomy group Spin(9) with finite volume given by
$V$. Assume that
 $$\mu_{1,p}(M)\geq\dfrac{22^p}{p^p}.$$
 Let us denote
$N(M)$ to be the number of ends (cusps) of $M$. Then  there exists a contant $C(n,p)>0$ depending only
on $n$ and $p$ such that,
$$N(M)\leq C(n,p)\Big(\frac{V}{V_{o}(1)}\Big)^2\ln^{\alpha(p)} \Big(\frac{V}{V_{o}(1)}\Big)$$
where $\alpha(p)=\max\{1, p/2\}$ and $V_o(1)$ denotes the volume of the unit ball centered at a
fixed point $o\in M$.
\end{theorem}
\section*{Acknowlegement}
A part of this paper was written during a stay of the first author
and the third author at the Vietnam Institute for Advanced Study in
Mathematics (VIASM). They would like to express their thanks to the
staff there for the hospitality and support.

\end{document}